\pdfoutput=1
\documentclass[12pt]{article}
\usepackage{amsmath,amssymb,amsfonts,here}
\usepackage{tikz}
\begin{document}
\newtheorem{theorem}{Theorem}[section]
\newtheorem{corollary}[theorem]{Corollary}
\newtheorem{lemma}[theorem]{Lemma}
\newtheorem{remark}[theorem]{Remark}
\newtheorem{example}[theorem]{Example}
\newtheorem{proposition}[theorem]{Proposition}
\newtheorem{definition}[theorem]{Definition}
\newtheorem{assumption}[theorem]{Assumption}
\def\emptyset{\varnothing}
\def\setminus{\smallsetminus}
\def\id{{\mathrm{id}}}
\def\G{{\mathcal{G}}}
\def\E{{\mathcal{E}}}
\def\H{{\mathcal{H}}}
\def\C{{\mathbb{C}}}
\def\N{{\mathbb{N}}}
\def\Q{{\mathbb{Q}}}
\def\R{{\mathbb{R}}}
\def\Z{{\mathbb{Z}}}
\def\Path{{\mathrm{Path}}}
\def\Str{{\mathrm{Str}}}
\def\st{{\mathrm{st}}}
\def\tr{{\mathrm{tr}}}
\def\opp{{\mathrm{opp}}}
\def\a{{\alpha}}
\def\be{{\beta}}
\def\de{{\delta}}
\def\e{{\varepsilon}}
\def\si{{\sigma}}
\def\la{{\lambda}}
\def\th{{\theta}}
\def\lan{{\langle}}
\def\ran{{\rangle}}
\def\isom{{\cong}}
\newcommand{\Hom}{\mathop{\mathrm{Hom}}\nolimits}
\newcommand{\End}{\mathop{\mathrm{End}}\nolimits}
\def\qed{{\unskip\nobreak\hfil\penalty50
\hskip2em\hbox{}\nobreak\hfil$\square$
\parfillskip=0pt \finalhyphendemerits=0\par}\medskip}
\def\proof{\trivlist \item[\hskip \labelsep{\bf Proof.\ }]}
\def\endproof{\null\hfill\qed\endtrivlist\noindent}

%%%%%%%%%%%%%%%%%%%%%%%%%%%%%%
\title{Flatness of $\alpha$-induced bi-unitary connections
and commutativity of Frobenius algebras}
\author{
{\sc Yasuyuki Kawahigashi}\\
{\small Graduate School of Mathematical Sciences}\\
{\small The University of Tokyo, Komaba, Tokyo, 153-8914, Japan}\\
{\small e-mail: {\tt yasuyuki@ms.u-tokyo.ac.jp}}
\\[0,40cm]
{\small iTHEMS Research Group, RIKEN}\\
{\small 2-1 Hirosawa, Wako, Saitama 351-0198,Japan}\\[0,05cm]
{\small and}
\\[0,05cm]
{\small Kavli IPMU (WPI), the University of Tokyo}\\
{\small 5--1--5 Kashiwanoha, Kashiwa, 277-8583, Japan}
}
\maketitle{}
\centerline{\sl Dedicated to the memory of Huzihiro Araki}

\begin{abstract}
The tensor functor called 
$\alpha$-induction produces a new unitary fusion category from a
Frobenius algebra, or a $Q$-system, in a braided
unitary fusion category.  A bi-unitary
connection, which is a finite family of complex number subject
to some axioms, realizes an object in any unitary fusion category.
It also gives a characterization of a finite-dimensional
nondegenerate commuting square in subfactor theory
of Jones and realizes a certain $4$-tensor 
appearing in recent studies of
$2$-dimensional topological order.  We study $\alpha$-induction
for bi-unitary connections, and show that flatness
of the resulting $\alpha$-induced bi-unitary connections implies
commutativity of the original Frobenius algebra.  
This gives a converse of our 
previous result and answers a question raised by R. Longo.
We furthermore give finer correspondence between the
flat parts of the $\alpha$-induced bi-unitary connections
and the commutative Frobenius subalgebras studied by
B\"ockenhauer-Evans.
\end{abstract}

\section{Introduction}

Jones theory of subfactors \cite{J}
has opened a new research direction, which is often called
\textsl{quantum symmetry} today, and this new type of symmetry
is described with a certain tensor category
\cite{BKLR}, \cite{EGNO}.  Such a tensor category is 
also a recent hot topic in physics literatures on
quantum field theory under
the name of \textsl{non-invertible symmetries}.

Roughly speaking, our tensor category has a $\mathbb{C}$-linear
space structure on the morphism spaces and the tensor product
operation on the objects.  In commonly used frameworks to
study tensor categories based on operator algebras, an object
is typically described with one of the following.

\begin{enumerate}
\item A bimodules over type II$_1$ factors
\item An endomorphism of a type III factor
\item A bi-unitary connection
\end{enumerate}

For detailed description of the tensor product operations and
morphism spaces, see \cite[Section 12.4]{EK2} for 1,
\cite{L1}, \cite{L2}, \cite{I} for 2, and
\cite[Section 3]{AH} for 3.  The approach 3 is not new at all,
but has caught much recent attention in relations to
tensor networks and 2-dimensional condensed matter physics
\cite{BMWSHV}, \cite{K4}, \cite{K5}, \cite{K7}.  This approach
has an advantage that everything can be computed within a
finite dimensional setting, while approaches 1 and 2 involve
infinite dimensional Hilbert spaces and operator algebras.
Also see \cite{BKLR} for a recent treatment of tensor categories
and subfactors, and \cite{EGNO} for a more abstract
and algebraic treatment of tensor categories.

Another important notion in subfactor theory
is a commuting square, which
originates in \cite{P1}.  We are particularly interested
in a \textsl{nondegenerate}
commuting square of finite dimensional $C^*$-algebras
with a trace as in \cite[Proposition 9.51]{EK2}.
Schou's thesis \cite{Sc} gave a highly useful characterization
of such a commuting square in terms of a \textsl{bi-unitary
connection} which generalizes a similar notion in \cite{O1}
used for a classification of subfactors.  See
\cite[Theorem 11.2]{EK2} for this characterization.

Let $A,B,C,D,X,Y$ be irreducible bimodules over
II$_1$ factors $N,M,P,Q$, such as ${}_M A_P$.
The composition $T_4(T_3\otimes_M \id_Y)
(\id_X\otimes_P T_1)^*T_2^*$ in the following diagram
gives an intertwiner from ${}_N D_Q$ to itself, 
so this gives a complex number, where we have
$T_1\in\Hom({}_M A\otimes_P Y_Q,{}_M B_Q)$,
$T_2\in\Hom({}_N X\otimes_M B_Q,{}_N D_Q)$,
$T_3\in\Hom({}_N X\otimes_M A_P, {}_N C_P)$,
$T_4\in{}_N C\otimes_P Y_Q,{}_N D_Q)$.

\begin{figure}[h]
\begin{center}
\begin{tikzpicture}[scale=3]
\draw [-to](1.5,1)--(1.75,1);
\draw [-to](1.5,2)--(1.68,2);
\draw [-to](1,1.8)--(1,1.2);
\draw [-to](2,1.8)--(2,1.2);
\draw (1,1)node{${}_N C\otimes_P Y_Q$};
\draw (0.85,2)node{${}_N X\otimes_M A\otimes_P Y_Q$};
\draw (2,1)node{${}_N D_Q$};
\draw (2.1,2)node{${}_N X\otimes_M B_Q$};
\draw (1.65,2.05)node[above]{$\id_X\otimes_P T_1$};
\draw (1.675,1)node[below]{$T_4$};
\draw (1,1.5)node[left]{$T_3\otimes_M\id_Y$};
\draw (2,1.5)node[right]{$T_2$};
\end{tikzpicture}
\label{diag}
\caption{The diagram for generalized quantum $6j$-symbols}
\end{center}
\end{figure}

A family of these complex numbers 
(up to some normalization and equivalence)
gives \textsl{generalized quantum $6j$-symbols} as in
\cite[Section 12.4]{EK2}.
(In \cite[Section 12.4]{EK2}, we have considered only 
two II$_1$ factors $N$ and $M$, but it is easy to extend
the arguments to the case involving four II$_1$ factors.
We say \textsl{generalized} because four algebras are
involved.)  We are interested in a subfactor $N\subset M$
which produces  finitely many irreducible $N$-$N$ bimodules
up to equivalence from tensor powers of ${}_N M_N$, 
and such a subfactor
is said to have a \textsl{finite depth}. As
long as we are interested in a subfactor of finite depth, all
bi-unitary connections we need are of the form in 
Fig.~\ref{diag} by
our result in \cite[Theorem 3.4]{K7}.
In today's understanding, these bi-unitary connections are
identified with $4$-tensors appearing in  
\cite{BMWSHV} as in \cite{K4}.

If we have $M=P$ in the above Fig.~\ref{diag}, a choice
${}_M A_M={}_M M_M$, which is the identity object for the tensor
product, is possible.
For this choice, a bi-unitary connection has a canonical
form and it is said to be \textsl{flat}.  A notion of a flat
bi-unitary connection was originally introduced in \cite{O1}, and
the current general form was studied in 
\cite[Theorem 2.1]{K1} in a purely algebraic context.
This flatness property is 
characterized by \cite[Theorem 11.17]{EK2}.  The name 
\textsl{flat} comes from its analogy to a notion of a flat
connection in differential geometry.  In this analogy, a bipartite
graph is regarded as a discrete analogue of a manifold.

A flat bi-unitary connection canonically arises from a subfactor
of finite depth through \textsl{higher relative commutants}.
This is used for classification
of subfactors.  In fact, if the original subfactor is 
\textsl{hyperfinite} and has finite index and finite depth, then
the celebrated classification theorem in
Popa \cite{P2} shows that this flat bi-unitary connection completely
recovers the original subfactor.

The next topic is a tensor functor called $\alpha$-induction
for a braided fusion category with a Frobenius algebra or a
$Q$-system.  This was first defined in  \cite{LR} and studied
in \cite{X} in a slightly different setting.  It was further
pursued in a more general setting in \cite{BE1}, \cite{BE2},
\cite{BE3}, and studied in a fully general setting in
\cite{BEK1}, \cite{BEK2} with a graphical method in \cite{O2}.
The original definition in \cite{LR} was closely related to
operator algebraic studies of chiral
conformal field theory as in \cite{KL}, \cite{KLM}, but
now this technique has been widely used in various different
context.  See \cite{H} for one recent example.

This $\alpha$-induction has been first studied in the setting
of endomorphisms of type III factors in 
\cite{BE1}, \cite{BE2}, \cite{BE3}, \cite{BEK1}, \cite{BEK2},
and it has been also studied for bimodules over type II$_1$ 
factors.  Our aim in our previous paper \cite{K8} was
to formulate this $\alpha$-induction in terms of
bi-unitary connections.
This paper is a continuation of \cite{K8} and our main result
is Theorem \ref{main}, which gives a positive answer to a
question of Longo on \cite{K8} and clarifies relations
between commutativity of a Frobenius algebra used for 
$\alpha$-induction and flatness of the $\alpha$-induced
bi-unitary connections.  Theorem \ref{corresp} gives
a further detailed understanding on these relations.

A chiral conformal field theory is behind all of these studies.
We refer a reader to \cite{EK3}, \cite{K2}, \cite{K3} for a survey
on relations among subfactor theory, conformal field theory and 
other topics in mathematical physics.

\section*{Acknowledgements}

This work was supported by 
Japan Science and Technology Agency (JST) as 
CREST program JPMJCR18T6 and
a part of Adopting Sustainable Partnerships for 
Innovative Research Ecosystem (ASPIRE), Grant Number JPMJAP2318,
and also by the National Science Foundation under Grant 
No.~DMS-1928930, while the author was in residence at the 
Simons Laufer Mathematical Sciences
Institute in Berkeley, California, during the year of 2024. 
A part of this work was done at Universit\`a di Roma,
``Tor Vergata'', Beijing Institute of Mathematical Sciences 
and Applications, and the Simons Laufer Mathematical Sciences
Institute.  The authors thanks S. Carpi and R. Longo 
for their hospitality at Universit\`a di Roma,
``Tor Vergata'' and Z. Liu at 
Beijing Institute of Mathematical Sciences and Applications.
The author also thanks R. Longo for his question at a conference
``Where Mathematics Meets Quantum Physics''
in Rome in June, 2023.

\section{Preliminaries on braided fusion categories, subfactors
and $\alpha$-induction}
\label{pre}

We give our setting for $\alpha$-induction for
a subfactor $N\subset M$ and endomorphisms of $N$
as in \cite{BEK1}, where $N$ and $M$ are type III factors.  
Throughout this paper, we have $\Delta$,
a finite set of mutually inequivalent
irreducible endomorphisms of $N$ of finite dimension,
as follows. (See \cite[Definition 2.1]{BEK1}.)

\begin{assumption}\label{assump0}{\rm
We have the following for $\Delta$.

(1) The identity automorphism is in $\Delta$.

(2) For any $\lambda\in\Delta$, we have another element
$\mu\in\Delta$ which is equivalent to $\bar\lambda$.

(3) For any $\lambda,\mu\in\Delta$, the composition
$\lambda\mu$ decomposes into a direct sum of irreducible
endomorphisms each of which is equivalent to one in $\Delta$.

(4) The set $\Delta$ has a \textsl{braiding} 
$\e(\lambda,\mu)\in\Hom(\lambda\mu,\mu\lambda)$ as in
\cite[Definition 2.2]{BEK1}.
}\end{assumption}

Note that the finiteness of the index of $\lambda\in\Delta$
implies that the identity automorphism is a direct summand of
$\lambda\bar\lambda$ by \cite[Theorem 4.1]{L2}.

\begin{definition}\label{bfc}{\rm
Let $\Sigma(\Delta)$ be the set of all endomorphisms of $N$
which decompose into finite direct sums of endomorphisms
that are unitarily equivalent to ones
in $\Delta$.  For $\lambda_1,\lambda_2\in\Sigma(\Delta)$,
we define the intertwiner space as
\[
\Hom(\lambda_1,\lambda_2)=\{t\in N\mid
t\lambda_1(x)=\lambda_2(x)t, \mathrm{\ for\ all\ }
x\in N\}.
\]
This gives a unitary rigid braided fusion category 
$\mathcal{C}$ with $\Sigma(\Delta)$ being the set of the objects, 
where the tensor product operation is given by the composition
of endomorphisms and the conjugate operation is given by
the conjugate endomorphism.
}\end{definition}

Note that all unitary braided fusion categories arise in this
way (for the unique injective type III$_1$ factor $N$, for
example).  This is because we can construct $6j$-symbols
from a braided fusion category and then construct a hyperfinite
type II$_1$ subfactor $N\subset M$ as in \cite[Section 13.3]{EK2},
and then we can make tensor products of $N$ and $M$ with
a common injective type III$_1$ factor.
Note that if the braiding
is nondegenerate in the sense of \cite[Definition 2.3]{BEK1},
this $\mathcal{C}$ is a unitary \textsl{modular tensor category},
but we do not need this nondegeneracy in this paper.

We further make an assumption on connectedness of certain
bipartite graphs as follows, which will be necessary in the
following Section.  This is assumed throughout the paper, except
for Section \ref{multi}.

\begin{assumption}\label{assump1}{\rm
We have $\mu\in\Sigma(\Delta)$ which satisfies
the following condition.
Consider a bipartite graph defined as follows.
Let both even and odd vertex sets be labeled with
the elements in $\Delta$.  The number of edges between
the even vertex $\nu_1$ and the odd vertex $\nu_2$ is
given by $\dim\Hom(\nu_1\mu,\nu_2)$.  We assume that this
bipartite graph is connected.
}\end{assumption}

In other words, this Assumption \ref{assump1}
means that for any $\nu\in\Delta$,
we have some positive integer $k$ so that
we have $\dim\Hom((\mu\bar\mu)^k,\nu)>0$.  We can always
achieve this by setting $\mu=\bigoplus_{\nu\in\Delta}\nu$.
Though irreducibility of $\mu$ was assumed in 
\cite[Assumption 2.2]{K8}, the construction and arguments
in \cite{K8} work without this irreducibility assumption.
We fix $\mu$ satisfying Assumption \ref{assump1} below.

We consider an irreducible 
subfactor $N\subset M$ with finite index whose canonical
endomorphism $\theta$ being in $\Sigma(\Delta)$.  
Such a subfactor automatically
has a finite depth.  For a fixed $N$ and
a fusion category of such endomorphisms, an irreducible
extension $M$ of $N$
is in a bijective correspondence to a \textsl{$Q$-system}
$(\theta,w,x)$ as in \cite[Chapter 3]{BKLR}, 
where $w\in\Hom(\id,\theta)$
and $x\in\Hom(\theta,\theta^2)$ satisfy the relations
as in \cite[Chapter 3]{BKLR}, and $\dim\Hom(\id,\theta)=1$.
(The original definition of a $Q$-system was given in 
\cite[Section 6]{L3} and had a redundant property.
One-dimensionality of $\Hom(\id,\theta)$ corresponds to
irreducibility of $M$.)
Such a $Q$-system gives a larger
factor $M\supset N$, the canonical endomorphism 
$\gamma:M\to M$, and an isometry $v\in M$ satisfying
$v\in\Hom(\id,\gamma)$, $M=Nv$ and $\gamma(v)=x$.
We say the $Q$-system $(\theta,w,x)$ is \textsl{local}
when we have $\e(\theta,\theta)x=x$.
If a modular tensor category arises as DHR-endomorphisms
of a completely rational conformal net as in \cite{KLM}, 
then a $Q$-system on it corresponds to an extension of
a conformal net, and this locality exactly corresponds
to locality of the extension as in the setting of
\cite[Theorem 4.9]{LR}.  This is why the
name locality is used here and 
it was called \textsl{chiral locality} in \cite{BEK1} to
emphasize chirality of a local conformal net.
In algebraic literature, this $Q$-system is
often called a \textsl{Frobenius algebra} and its locality
is called \textsl{commutativity}. 
We also use these names in this paper. (In this paper, we
only consider $C^*$-tensor categories. So our Frobenius
algebra in this paper is what is called $C^*$ Frobenius algebra in 
\cite{BKLR}.  Also see \cite[Theorem 2.3]{M}
for a bimodule formulation of a $Q$-system.)

Fix a (not necessarily commutative) Frobenius algebra $(\theta,w,x)$
and consider the corresponding $M\supset N$ and the inclusion map
$\iota:N\hookrightarrow M$.

The procedure called
\textsl{$\alpha$-induction} was defined in \cite[Proposition 3.9]{LR} 
as follows.
\[
\alpha_\lambda^\pm=\bar\iota^{-1}\cdot
\mathrm{Ad}(\varepsilon^\pm(\lambda,\theta))\cdot\lambda\cdot\bar\iota,
\]
where $\lambda$ is an endomorphism of $N$ in $\Sigma(Delta)$, 
$\theta=\bar\iota\cdot\iota$
is the dual canonical endomorphism of $N\subset M$, and $\pm$
stands for a choice of positive/negative braiding.
It is a nontrivial fact that
$\mathrm{Ad}(\varepsilon^\pm(\lambda,\theta))\cdot
\lambda\cdot\bar\iota(x)$ is in the image of $\bar\iota$ 
for $x\in M$.  We have $\alpha_\lambda^\pm(x)=\lambda(x)$ for
$x\in N$ and $\alpha_\lambda^\pm(v)=
\varepsilon^\pm(\lambda,\theta)^*v$.
See \cite[Section 3]{BEK1} for basic properties 
of $\alpha$-induction in this setting.

We construct a bi-unitary connection $W_4(\alpha_\lambda^+,\mu)$
as follows.  (See \cite[Section 3]{K8} for a general definition
of a bi-unitary connection.)

\begin{definition}\label{W4}{\rm
Let $a_1,a_2,a_3,a_4$ be irreducible $M$-$N$ morphisms
arising from $\Delta$ and the subfactor $N\subset M$.
Consider the diagram in Fig.~\ref{conn2x}.
By composing isometries $T_1\in\Hom(a_1\mu,a_2)$,
$T_2\in\Hom(\alpha_\lambda^+ a_2,a_4)$,
$T_3\in\Hom(\alpha_\lambda^+ a_1,a_3)$, and
$T_4\in\Hom(a_3\mu,a_4)$, we obtain a complex number
$T_4 T_3 \alpha_\lambda^+(T_1^*)T_2^*\in\Hom(a_4,a_4)$.
We define the connection $W_4(\alpha^+_\lambda,\mu)$ by 
this number and represent this as in Fig.~\ref{conn11}, where
labels for edges are dropped.
}\end{definition}

\begin{figure}[H]
\begin{center}
\begin{tikzpicture}[scale=2.5]
\draw [-to](1.2,1)--(1.8,1);
\draw [-to](1.3,2)--(1.7,2);
\draw [-to](1,1.8)--(1,1.2);
\draw [-to](2,1.8)--(2,1.2);
\draw (1,1)node{$a_3\mu$};
\draw (2,1)node{$a_4$};
\draw (1,2)node{$\alpha_\lambda^+ a_1\mu$};
\draw (2,2)node{$\alpha_\lambda^+ a_2$};
\draw (1.5,2)node[above]{$\alpha_\lambda^+(T_1)$};
\draw (1.5,1)node[below]{$T_4$};
\draw (1,1.5)node[left]{$T_3$};
\draw (2,1.5)node[right]{$T_2$};
\end{tikzpicture}
\caption{The diagram for the connection $W_4(\alpha_\lambda^+,\mu)$}
\label{conn2x}
\end{center}
\end{figure}

\begin{figure}[H]
\begin{center}
\begin{tikzpicture}%[scale=0.8]
\draw [-to](1,1)--(2,1);
\draw [-to](1,2)--(2,2);
\draw [-to](1,2)--(1,1);
\draw [-to](2,2)--(2,1);
\draw (1,1)node[below left]{$a_3$};
\draw (2,1)node[below right]{$a_4$};
\draw (1,2)node[above left]{$a_1$};
\draw (2,2)node[above right]{$a_2$};
\end{tikzpicture}
\caption{The standard diagram for a connection $W_4(\alpha^+_\lambda,\mu)$}
\label{conn11}
\end{center}
\end{figure}

That is, all the four vertex sets are labeled with represenatives
of the irrducible $M$-$N$ morphisms arising from $\Delta$ and
$\iota:N\hookrightarrow M$.  For such irrducible $M$-$N$ morphisms
$a, b$, the number of edges between $a$ and $b$ for the both
horizontal graphs is given by $\dim\Hom(a\mu,b)$ and
the number of edges between $a$ and $b$ for the both
vertical graphs is given by $\dim\Hom(\alpha^+_\lambda a,b)$
The complex numbers arising in this way satisfy
\textsl{bi-unitarity} axiom \cite[Definition 2.2]{K5}.

Note that we take $T_1,T_2,T_3,T_4$ from orthonormal bases of
these $\Hom$ spaces, where we have natural inner products.
We use the same $T_1$, $T_4$ for all such bi-unitary connections.
The bi-unitary connection values have ambiguity depending of
choices of $T_2$, $T_3$.  Different such choices give
equivalent bi-unitary connections.  (See \cite[Section 3]{AH}.)

Connectivity of the horizontal graph for this
bi-unitary connection is equivalent to the condition that
for any irreducible $M$-$N$ morphisms $a,b$, we have a
positive integer $k$ satisfying
$\dim\Hom(a(\mu\bar\mu)^k,b)>0$.  This holds true because we have a
positive integer $k$ satisfying $\dim\Hom((\mu\bar\mu)^k,\bar a b)>0$
by Assumption \ref{assump1}.  Note that connectivity of
the vertical graphs are now not required.

This definition is logically fine, but rather unsatisfactory
because it requires knowledge on intertwiners spaces of
$\alpha^+_\lambda$ as an endomorphism, while we intend
to construct a bi-unitary connection corresponding to
$\alpha^+_\lambda$.  As shown in
\cite[Section 4]{K8}, we first draw Fig.~\ref{conn12}
and then can apply some topological moves
to remove this problem as in 
Fig.~\ref{conn13}, which now involves only intertwiners among
$N$-$N$ and $M$-$N$ morphisms.  
(See \cite[Section 3]{BEK1} for convention of graphical
calculus we use.  Note that we compose intertwiners in the
direction from the top to the bottom, while some authors
use the opposite convention.)
The complex number represented
with this diagram gives the value of the $\alpha$-induced
bi-unitary connection $W_4(\alpha^+_\lambda,\mu)$.

\begin{figure}[H]
\begin{center}
\begin{tikzpicture}[scale=0.6]
\draw [thick](6,13.5) arc (0:180:1.75);
\draw [thick](2.5,13.5) arc (90:180:1.5);
\draw (4,12) arc (0:90:1.5);
\draw (3.5,2.5) arc (270:360:1.5);
\draw (5,7) arc (0:90:1);
\draw (2,4) arc (270:360:1);
\draw [thick](4,8) arc (90:180:1);
\draw [thick](2,4) arc (180:270:1.5);
\draw [thick](3.5,2.5) arc (180:360:1.25);
\draw [thick](1,5) arc (180:270:1);
\draw [thick](1,5)--(1.9,5.9);
\draw [thick](2.1,6.1)--(3,7);
\draw [thick](1,12)--(2.4,10.6);
\draw [thick](2.6,10.4)--(4,9);
\draw [ultra thick](2,6)--(1,7);
\draw [ultra thick](1,7)--(1,9);
\draw [ultra thick](1,9)--(2.5,10.5);
\draw (2.5,10.5)--(4,12);
\draw (5,4)--(5,7);
\draw [thick](4,8)--(4,9);
\draw (3,5)--(2,6);
\draw [thick](6,2.5)--(6,13.5);
\draw (4,13.5)node{$\lambda$};
\draw (1,13.5)node{$a_2$};
\draw (1,8)node[left]{$\alpha_\lambda^+$};
\draw (1,4)node{$a_1$};
\draw (3,4)node{$\lambda$};
\draw (5,4)node[left]{$\mu$};
\draw (2,2.5)node{$a_3$};
\draw (3.5,1.25)node{$a_4$};
\end{tikzpicture}
\caption{The diagram for the connection in Fig.~\ref{conn2x}}
\label{conn12}
\end{center}
\end{figure}

\begin{figure}[H]
\begin{center}
\begin{tikzpicture}[scale=0.8]
\draw [thick](5,7) arc (0:180:1);
\draw [thick](3,7) arc (90:180:1);
\draw [thick](2,6) arc (90:180:1);
\draw [thick](1,4) arc (180:270:1);
\draw [thick](2,3) arc (180:270:1);
\draw [thick](3,2) arc (180:360:1);
\draw [thick](1,4)--(1,5);
\draw [thick](5,2)--(5,7);
\draw (4,6) arc (0:90:1);
\draw (3,5) arc (0:90:1);
\draw (2,3) arc (270:360:1);
\draw (3,2) arc (270:360:1);
\draw (4,3)--(4,4);
\draw (4,5)--(4,6);
\draw (3,4)--(4,5);
\draw (3,5)--(3.4,4.6);
\draw (3.6,4.4)--(4,4);
\draw (3,8)node{$a_4$};
\draw (2,7)node{$a_2$};
\draw (1,4.5)node[left]{$a_1$};
\draw (4,3.5)node[right]{$\mu$};
\draw (4,5.5)node[right]{$\lambda$};
\draw (2,2)node{$a_3$};
\end{tikzpicture}
\caption{Redrawing of Fig.~\ref{conn12} with topological moves}
\label{conn13}
\end{center}
\end{figure}

We investigate flatness of the $\alpha$-induced
bi-unitary connection $W_4(\alpha^+_\lambda,\mu)$
in this paper.

\section{Flatness of $\alpha$-induced bi-unitary connections}
\label{flat}

We choose $\mu$ as in Assumption \ref{assump0} and
consider $\alpha$-induced bi-unitary connections
$W_4(\alpha^\pm_\lambda,\mu)$ as in the previous Section.
For simplicity of notations, we use only the positive induction
$\alpha^+$ in this Section, but the case of $\alpha^-$ is
completely parallel, needless to say.
We study an issue of flatness of the bi-unitary connections
$W_4(\alpha^\pm_\lambda,\mu)$.

We now present our main Theorem in this paper.
The ``if'' part of the following Theorem
is \cite[Theorem 5.2]{K8}, and the new ``only if'' part gives
a converse.  This converse part gives an answer to
a question of R. Longo made in the conference 
``Where Mathematics Meets Quantum Physics ''
in the summer of 2023 at Rome.  We also give a new proof
for the ``if'' part, because we need this argument for
the ``only if'' part.
Note that commutativity of the Frobenius algebra, or locality
of the $Q$-system, was assumed in all of \cite[Section 5]{K8}.

\begin{theorem}\label{main}
We fix an object $\mu$ in the original braided
fusion category $\mathcal{C}$
as in Assumptions \ref{assump0} and \ref{assump1} and a Frobenius
algebra $(\theta,w,x)$ in $\mathcal{C}$.  Then
$\alpha$-induced bi-unitary connections $W_4(\alpha^+_\lambda,\mu)$
arising from the Frobenius algebra
are flat with respect to the initial vertex $\iota$
for all $\lambda$ if and only if the Frobenius
algebra is commutative, or the $Q$-system is local.
\end{theorem}

\begin{proof}
We first give a proof for the ``if'' part.
For any $\lambda\in\Delta$ and a nonnegative 
integer $k$, we have a natural inclusion
\begin{equation}
\label{ineq1}
\End((\bar\alpha_\lambda^+\alpha_\lambda^+)^k)
\subset \End((\bar\alpha_\lambda^+\alpha_\lambda^+)^k\iota),
\end{equation}
and 
\begin{equation}
\label{ineq1p}
\End(\alpha_\lambda^+(\bar\alpha_\lambda^+\alpha_\lambda^+)^k)
\subset \End(\alpha_\lambda^+
(\bar\alpha_\lambda^+\alpha_\lambda^+)^k\iota),
\end{equation}
and by \cite[Theorem 3.3]{S}, we know
that $\End((\bar\alpha_\lambda^+\alpha_\lambda^+)^k)$ 
and $\End(\alpha_\lambda^+(\bar\alpha_\lambda^+\alpha_\lambda^+)^k)$
give the flat part of this $\alpha$-induced bi-unitary connection.
For the dimensions of these endomorphism spaces, we further have
the following inequality.
\begin{align}
\label{ineq2}
\dim\End((\bar\alpha_\lambda^+\alpha_\lambda^+)^k)
&=
\dim\End((\alpha^+_{(\bar\lambda\lambda)^k})\notag\\
&\le\dim\End(\alpha^+_{(\bar\lambda\lambda)^k}\iota)\notag\\
&=\dim\End(\iota(\bar\lambda\lambda)^k)\notag\\
&=\dim\Hom(\theta(\bar\lambda\lambda)^k,(\bar\lambda\lambda)^k),
\end{align}
and
\begin{align}
\label{ineq2p}
\dim\End(\alpha_\lambda^+(\bar\alpha_\lambda^+\alpha_\lambda^+)^k)
&=
\dim\End((\alpha^+_{\lambda(\bar\lambda\lambda)^k})\notag\\
&\le\dim\End(\alpha^+_{\lambda(\bar\lambda\lambda)^k}\iota)\notag\\
&=\dim\End(\iota(\lambda\bar\lambda\lambda)^k)\notag\\
&=\dim\Hom(\theta\lambda(\bar\lambda\lambda)^k,
\lambda(\bar\lambda\lambda)^k),
\end{align}
because $\alpha^+_{(\bar\lambda\lambda)^k}$ and
$\alpha^+_{\lambda(\bar\lambda\lambda)^k}$ 
are endomorphisms of $M$ extending the 
endomorphisms $\lambda(\bar\lambda\lambda)^k$ and
$(\bar\lambda\lambda)^k$ of $N$, respectively.

Recall that for $\lambda_1,\lambda_2\in\Sigma(\Delta)$, we have 
\begin{equation}
\label{ineq3}
\dim\Hom(\alpha^+_{\lambda_1},\alpha^+_{\lambda_2})\le
\dim\Hom(\theta\lambda_1,\lambda_2)
\end{equation}
by \cite[Theorem 3.9]{BE2}, and we have an equality here
if we have commutativity of the Frobenius
algebra, or locality of the $Q$-system.
(As explained in the first two lines of \cite[page 455]{BEK1},
this inequality does not require commutativity of the Frobenius
algebra, but equality does.)

If we have commutativity of the Frobenius
algebra, or locality of the $Q$-system, then we have
an equality in (\ref{ineq3}), and this gives equalities
in (\ref{ineq2}) and (\ref{ineq2p}), which in turn give
equalities in (\ref{ineq1}) and (\ref{ineq1p}).
Because $\End((\bar\alpha_\lambda^+\alpha_\lambda^+)^k)$ 
and $\End(\alpha_\lambda^+(\bar\alpha_\lambda^+\alpha_\lambda^+)^k)$
give the flat parts of our $\alpha$-induced bi-unitary connection
as above, these equalities mean that 
the $\alpha$-induced bi-unitary connection
$W_4(\alpha^+_\lambda,\mu)$ is flat.

We now prove the converse.
If we have flatness of $W_4(\alpha^+_\lambda,\mu)$ for
all $\lambda\in\Delta$, we have equalities in (\ref{ineq1}),
(\ref{ineq2}), and (\ref{ineq2p}).  
For all $\lambda_1,\lambda_2\in\Delta$
appearing in the irreducible decompositions of
$(\bar\lambda\lambda)^k$ we have 
\begin{equation}
\label{ineq4}
\dim\Hom(\alpha^+_{\lambda_1},\alpha^+_{\lambda_2})\le
\dim\Hom(\theta\lambda_1,\lambda_2)
\end{equation}
as above, but equalities in (\ref{ineq2}) and (\ref{ineq2p})
mean that we have an equality in (\ref{ineq4}) for all such
$\lambda_1,\lambda_2$.  Assumption \ref{assump1} implies that
all $\lambda_1,\lambda_2\in\Delta$ arise in this way
if we choose $\lambda$ to be $\mu$ in Assumption \ref{assump1}.
This in particular implies
\begin{equation}
\label{ineq5}
\dim\Hom(\alpha^+_{\lambda},\alpha^+_{\id})=
\dim\Hom(\theta\lambda,\id)=
\dim\Hom(\lambda,\theta).
\end{equation}

Recall that the modular invariant matrix $Z$ is
defined as $Z_{\lambda_1,\lambda_2}=
\dim\Hom(\alpha^+_{\lambda_1},\alpha^-_{\lambda_2})$
as in \cite[Definition 5.5]{BEK1}.
This gives that we now have
\[
Z_{\lambda,0}=\dim\Hom(\alpha^+_{\lambda},\alpha^-_{\id})=
\dim\Hom(\alpha^+_{\lambda},\alpha^+_{\id})
\dim\Hom(\lambda,\theta),
\]
where $0$ denotes the identity automorphism, and
this in turn gives commutativity of the Frobenius
algebra, or locality of the $Q$-system by
\cite[Proposition 3.2]{BE4}.
\end{proof}

We further continue this study on relations between
flatness of $\alpha$-induced bi-unitary connections
$W_4(\alpha_\lambda^+,\mu)$
and commutativity of the Frobenius algebra $(\theta, w,x)$,
or locality of the $Q$-system.
We do not assume 
commutativity of the Frobenius algebra $(\theta, w,x)$ here.

As in \cite[Theorem 4.7]{BE4}, we have an intermediate
subfactor $N\subset M_+\subset M$ where the dual canonical
endomorphism $\theta_+\in\Sigma(\Delta)$ arising from $N\subset M_+$ is
given by $\bigoplus_{\lambda\in\Delta} Z_{\lambda,0}\lambda$,
where the modular invariant matrix $Z$ is given by
$Z_{\lambda_1,\lambda_2}=
\dim\Hom(\alpha^+_{\lambda_1},\alpha^-_{\lambda_2})$
as in \cite[Definition 5.5]{BEK1}.
This is a ``commutative part'' of the original Frobenius
algebra $(\theta,w,x)$.
We write $\tilde\alpha^+_{+;\lambda}$ for a positive
$\alpha$-induction of $\lambda\in\Sigma(\Delta)$ from
$N$ to $M_+$ as in \cite[Section 4, page 276]{BE4}.

Then \cite[Lemma 5.1]{BE4} gives
\[
\End(\alpha^+_{{(\bar\lambda\lambda)}^k})=
\End(\tilde\alpha^+_{+;{(\bar\lambda\lambda)}^k}),
\]
and the left-hand side gives the flat part of 
$W_4(\alpha^+_\lambda,\mu)$ as in the proof of Theorem \ref{main}
Now the right-hand side of the above inclusion is equal to
$\End(\tilde\alpha^+_{+;{(\bar\lambda\lambda)}^k} \iota_+)$
where $\iota_+$ is the inclusion map from $N$ to $M_+$
because the commutativity of the Frobenius algebra
for the inclusion $N\subset M_+$ gives flatness of the 
$\alpha$-induced bi-unitary connection as in Theorem \ref{main}.
That is, the flat parts of the $\alpha$-induced
bi-unitary connection $W_4(\alpha^+_\lambda,\mu)$ 
given by the intertwiner spaces of the flat bi-unitary
connections $W_4(\tilde\alpha^+_{+,\lambda},\mu)$ 
corresponding to the commutative part Frobenius algebra
arising from $\theta_+$ as in \cite[Theorem 4.7]{BE4}.
We have thus obtained the following Theorem.

\begin{theorem}\label{corresp}
We fix an irreducible object $\mu$ in the original braided
fusion category $\mathcal{C}$
as in Assumptions \ref{assump0} and \ref{assump1} and a Frobenius
algebra $(\theta,w,x)$ in $\mathcal{C}$.  Then the flat parts
of the $\alpha$-induced bi-unitary connections $W_4(\alpha_\lambda^+,\mu)$
are given by the intertwiner spaces of the flat bi-unitary
connections $W_4(\tilde\alpha^+_{+,\lambda},\mu)$ 
corresponding to the commutative part Frobenius algebra
arising from $\theta_+$.
\end{theorem}

The meaning of this Theorem is understood as follows.
We start with a not necessarily commutative Frobenius algebra
$(\theta,w,x)$,
apply the $\alpha^+$-induction to get a bi-unitary connection,
and look at it flat parts, which are given as the commutants 
within $N$.  In another way, we first look at the commutative
part Frobenius algebra arising from $\theta_+$, apply 
the $\alpha^+$-induction to get a flat bi-unitary connection,
and look at the intertwiner spaces it gives.  The two
resulting finite dimensional spaces are the same.

\section{$\mathbb{Z}/n{\mathbb{Z}}$-grading and multi-fusion categories}
\label{multi}

As pointed out in \cite[Section 6]{K8},
Assumption \ref{assump1} is not satisfied for any
\textsl{irreducible} $\mu$
even for the most well-studied modular tensor category
$SU(2)_k$ arising from the Wess-Zumino-Witten models
due to the $\mathbb{Z}/2{\mathbb{Z}}$-grading.
(See  \cite[Subsection 16.2.3]{DMS}
for the Wess-Zumino-Witte models.)
We study this issue and extend considerations in the previous
Section based on ideas in \cite[Section 6]{K8}.

Let $n$ be a positive integer larger than $1$.  By a
$\mathbb{Z}/n{\mathbb{Z}}$-grading on $\Delta$, we mean the 
following.

\begin{assumption}\label{assump2}{\rm
We have a disjoint union 
$\Delta=\bigcup_{j\in\mathbb{Z}/n{\mathbb{Z}}}\Delta_j$
so that for $\lambda_1\in \Delta_j$ and $\lambda_2\in\Delta_k$,
the composition $\lambda_1\lambda_2$ decomposes into 
a direct sum of elements in $\Delta_{j+k}$.
}\end{assumption}

We now assume the above Assumption \ref{assump2}, where
each $\Delta_j$ is non-empty.  Note that we have $\id\in\Delta_0$
and if we have $\lambda\in\Delta_j$, then we have 
$\bar\lambda\in\Delta_{-j}$.  It is well-known that the modular
tensor category arising from the Wess-Zumino-Witten model
$SU(n)_k$ has a $\mathbb{Z}/n{\mathbb{Z}}$-grading, where $k$
is a positive integer.

Consider another Assumption as follows.

\begin{assumption}\label{assump3}{\rm
We have $\mu\in\Delta_k$ satisfying the following property.
For any $j\in\mathbb{Z}/n{\mathbb{Z}}$, consider
a bipartite graph whose even and odd vertex sets are labeled with
the elements in $\Delta_j$ and $\Delta_{j+k}$, respectively.
The number of edges between
the even vertex $\nu_1$ and the odd vertex $\nu_2$ is
given by $\dim\Hom(\nu_1\mu,\nu_2)$.  We assume that this
bipartite graph is connected.
}\end{assumption}

Note that irreducibility of $\mu$ is now assumed.
As in the case of Assumption \ref{assump1}, 
this property of $\mu$ is equivalent to the following.
For any $\nu_1,\nu_2\in\Delta_j$, we have some positive integer
$l$ that $\dim\Hom(\nu_1(\mu\bar\mu)^l,\nu_2)=
\dim\Hom((\mu\bar\mu)^l,\bar\nu_1\nu_2)>0$.
Thus, if the irreducible decomposition of
some power $(\mu\bar\mu)^l$ contains all endomorphisms in $\Delta_0$,
this Assumption is automatically satisfied.
We fix $\mu$ which satisfies this Assumption \ref{assump3}.

Consider a Frobenius algebra $(\theta,w,x)$ with 
$\theta\in\Sigma(\Delta_0)$, the corresponding
subfactor $N\subset M$ and the inclusion map 
$\iota:N\hookrightarrow M$ so that we have $\th=\bar\iota\iota$.
Let $\Phi_j$ be the set of representatives of 
irreducible unital $*$-algebra homomorphisms
from $N$ to $M$ arising from irreducible decompositions
of homomorphisms in $\{\iota\lambda\mid\lambda\in\Delta_j\}$
for $j\in\mathbb{Z}/n{\mathbb{Z}}$.  All the
irreducible $M$-$N$ morphisms arising from $\iota$
and $\Delta$ appear in irreducible decomposition of
morphisms in 
\[
\{(\iota\bar\iota)^l\iota\lambda\mid l\in\N, \lambda\in\Delta_j\}
=\{(\iota\th^l\lambda\mid l\in\N, \lambda\in\Delta_j\}
\]
by definition, so the union 
$\bigcup_{j\in\mathbb{Z}/n{\mathbb{Z}}}\Phi_j$ covers all
such morphisms up to unitary equivalence.
We now show $\Phi_j$'s are mutually disjoint.
Suppose the same irreducible $M$-$N$ morphism $a$ belongs
to $\Phi_j$ and $\Phi_m$.  Then $\id$ appears in the
irreducible decomposition of $\bar aa$, and because
of $\th\in\Sigma(\Delta_0)$, we know that $\id$ belongs
to $\Delta_{j-m}$, which implies $j=m\in\mathbb{Z}/n\mathbb{Z}$.

\begin{definition}\label{gj}
{\rm  
Define a bipartite graph $\G_j$ for $k\in \mathbb{Z}/n\mathbb{Z}$
as follows. The even and odd vertex sets are labeled with
the elements in $\Phi_j$ and $\Phi_{j+k}$, respectively.
The number of edges between
the even vertex $a$ and the odd vertex $b$ is
given by $\dim\Hom(a\mu,b)$. 
}\end{definition}

By Assumption \ref{assump3}, this graph is connected.
Two graphs $\G_j$ and $\G_m$ might happen to be the same
for different $j$ and $m$, but we distinguish them and
regard them as different graphs now.

Take $\lambda\in\Delta_j$.  We construct a bi-unitary
connection $W_4(\alpha^+_\lambda,\mu)$ as in 
Section \ref{pre}, but
we use only irreducible $M$-$N$ morphisms in $\Delta_0$
for the upper left vertices and only those in $\Delta_j$
for the lower left vertices.  Then the top horizontal
graph of $W_4(\alpha^+_\lambda,\mu)$ is $\G_0$ and
the bottom one is $\G_j$.

Recall that we can compose two bi-unitary connections 
only when the bottom horizontal graph of the former matches
the top horizontal graph of the latter.  By defining the
composition to be zero if they do not match, we consider
a unitary multi-fusion category generated by bi-unitary
connections $W_4(\alpha^+_\lambda,\mu)$.  Now the top
and bottom horizontal graphs of any bi-unitary connection
arising in this way are among $\G_j$'s, so the identity
object in this multi-fusion category is decomposed into
$n$ irreducible objects.  Then in this setting, the
same arguments in the proofs of Theorems \ref{main2}
and \ref{corresp2} work, so we obtain the following two
Theorems.  Example \ref{ex1} gives a concrete explanation
of this setting.

\begin{theorem}\label{main2}
We fix an irreducible object $\mu$ in the original braided
fusion category $\mathcal{C}$
as in Assumptions \ref{assump0}, \ref{assump2}
and \ref{assump3} and a Frobenius
algebra $(\theta,w,x)$ in $\mathcal{C}$ as above.  Then
$\alpha$-induced bi-unitary connections $W_4(\alpha^+_\lambda,\mu)$
arising from the Frobenius algebra
are flat with respect to the initial vertex $\iota$
for all $\lambda$ if and only if the Frobenius
algebra is commutative, or the $Q$-system is local.
\end{theorem}

\begin{theorem}\label{corresp2}
We fix an irreducible object $\mu$ in the original braided
fusion category $\mathcal{C}$
as in Assumptions \ref{assump0}, \ref{assump2}
and \ref{assump3} and a Frobenius
algebra $(\theta,w,x)$ in $\mathcal{C}$.  Then the flat parts
of the $\alpha$-induced bi-unitary connections $W_4(\alpha_\lambda^+,\mu)$
are given by the intertwiner spaces of the flat bi-unitary
connections $W_4(\tilde\alpha^+_{+,\lambda},\mu)$ 
corresponding to the commutative part Frobenius algebra
arising from $\theta_+$ as in Theorem \ref{corresp}.
\end{theorem}

\begin{example}\label{ex1}{\rm
Consider the modular tensor category $\mathcal{C}$ arising
from the Wess-Zumino-Witten model $SU(2)_{2l}$ and label
the representatives of its irreducible objects as
$0,1,\dots,2l$ as usual.  We realize this modular tensor
category as one consisting of endomorphisms of a type
III factor $N$.  So now the labels $0,1,\dots,2l$ are
for irreducible endomorphisms of $N$.

Let $n$=2, $\Delta_0=\{0,2,\dots,2l\}$,
$\Delta_1=\{1,3,\dots,2l-1\}$, $\mu=1$, and $k$=1 in 
the above setting in this Section.  Then Assumptions
\ref{assump0}, \ref{assump1} and \ref{assump2} are
satisfied.

Then the category we obtain from such bi-unitary connections
is not the original modular tensor category $\mathcal{C}$,
but a multi-fusion category as follows.  Because of disconnectedness
of the graph arising from the choice $\mu=1$, we have two types
of the horizontal graphs.  One graph $\mathcal{G}_0$
connects the vertices $\Delta_0$
to $\Delta_1$, and the other $\mathcal{G}_1$ is the 
reversed graph.  These two
are isomorphic graphs, but we distinguish them.  Because the
composition of two bi-unitary connections is possible only when
the two horizontal graphs match.  Note that 
we have four types of bi-unitary
connections as follows.
(For the standard generator
$\lambda=1$, the top graph is $\mathcal{G}_0$ and the bottom
graph is $\mathcal{G}_1$.)

\begin{enumerate}
\item The top graph is $\mathcal{G}_0$ and the bottom
graph is $\mathcal{G}_1$.
\item The top graph is $\mathcal{G}_0$ and the bottom
graph is $\mathcal{G}_0$.
\item The top graph is $\mathcal{G}_1$ and the bottom
graph is $\mathcal{G}_0$.
\item The top graph is $\mathcal{G}_1$ and the bottom
graph is $\mathcal{G}_1$.
\end{enumerate}

We can definie the composition is 0 if the two graphs do not
match.  In this way, we obtain a multi-fusion category
where the identity object decomposes into 2 components.
This reproduces a multi-fusion category consisting of
4 types of bimodules: $N$-$N$, $N$-$M$, $M$-$N$ and $M$-$M$ 
arising from a subfactor $N\subset M$ with principal
graph $A_{2l+1}$.

Choose a Frobenius algebra $(\theta,w,x)$ in
$\mathcal{C}$ corresponding to the Goodman-de la Harpe-Jones
subfactors realized by one of the $A$-$D$-$E$ Dynkin diagrams
\cite[Section 4.5]{GHJ}.
That is, besides trivial case $\th=\id$, we have
$\th=0\oplus 2l$ in $SU(2)_{2l}$ corresponding to the $D_{l+2}$ case,
$\th=0\oplus6$ in $SU(2)_{10}$ corresponding to the $E_6$ case,
$\th=0\oplus8\oplus16$ in $SU(2)_{16}$ corresponding to the $E_7$ case,
and $\th=0\oplus10\oplus18\oplus20$ 
in $SU(2)_{10}$ corresponding to the $E_8$ case.
The case of even $l$ for $D_{l+2}$, odd $l$ for $D_{l+2}$,
$E_6$, $E_7$, and $E_u$ have been studied in
\cite[Section 6.2]{BE3},
\cite[Section 5.2]{BEK2},
\cite[Section 6.1]{BE3},
\cite[Section 5.3]{BEK2}, and
\cite[Section 6.1]{BE3}, respectively.  

Then the resulting $\alpha$-induced bi-unitary connections
$W_4(\alpha^\pm_\lambda,\mu)$ are the usual bi-unitary
connections on the $A$-$D$-$E$ Dynkin diagrams 
\cite[Fig.~11.32]{EK2}. Theorem \ref{main2}
is then simply a well-known fact on the correspondence between
flatness of bi-unitary connections and commutativity of
the Frobenius algebras, or locality of the $Q$-systems.
(Local $Q$-systems here correspond to local extensions
of the conformal nets arising from $SU(2)_{2l}$.  They
are simple current extensions of order 2 and conformal
embeddings $SU(2)_{10}\subset SO(5)_1$ and
$SU(2)_{28}\subset (G2)_1$.  See \cite[Theorem 2.4]{KL}.)

It is known that the bi-unitary connections on
$E_7$ have flat parts arising from the $D_{10}$ diagram as in
\cite{EK1}.  The intermediate subfactors $M_\pm$ 
arising from the $E_7$ modular invariant matrix $Z$ in this
setting as in \cite[Theorem 4.7]{BE4} are given by
$\th_\pm=0\oplus 16$, which also gives a local simple
current extension of order 2.  Thus the commutative
Frobenius algebras arising from $\th_\pm$ as in
\cite[Theorem 4.7]{BE4} also give the $D_{10}$ diagram.
This is an example of Theorem \ref{corresp2}.
}\end{example}

A family of Examples \ref{ex1} has been already well-understood.
Our Theorems \ref{main2} and \ref{corresp2} show
that this correspondence between flat bi-unitary
connections and commutative Frobenius algebras is not a
simple coincidence, but logical and general necessity.

\end{document}